\documentclass{article}
\usepackage[utf8]{inputenc}

\usepackage{amsmath}
\usepackage{amsthm}
\usepackage{amssymb,amsfonts}
\usepackage{hyperref}
\usepackage[capitalize]{cleveref}
\usepackage{latexsym}
\usepackage{todonotes}
\usepackage{nicefrac}
\usepackage{epsfig}
\usepackage{stmaryrd}
\usepackage{float}
\usepackage{setspace}
\usepackage{tikz-cd}
\usepackage{enumerate}
\usepackage[all]{xypic}
\usepackage{bbm,ifpdf,tikz}
\ifpdf
\usepackage{pdfsync}
\fi
\usetikzlibrary{positioning,matrix,arrows,decorations.pathmorphing}
\usepackage{enumitem}

\usetikzlibrary{positioning,arrows}
\usetikzlibrary{decorations.markings}
\tikzstyle{chamberlabel} = [draw, teal!60!gray, shape=circle, minimum size=5pt, inner sep=1pt]
\tikzstyle{invisivertex} = [black, shape=rectangle, minimum size=0pt, inner sep=2pt]
\tikzstyle{point}=[draw, black, fill,shape=circle, minimum size=4pt, inner sep=0pt]
\tikzstyle{moebius} = [draw, teal, shape=circle, minimum size=8pt, inner sep=1.5pt]

\newtheorem{thm}{Theorem}
\newtheorem{prop}[thm]{Proposition}

\newtheorem{lemma}[thm]{Lemma}

\theoremstyle{definition}
\newtheorem{definition}{Definition}

\newtheorem{rmk}{Remark}

\theoremstyle{definition}

\newcommand{\excise}[1]{}

\newcommand{\R}{\mathbb{R}}

\renewcommand{\cH}{\mathcal{H}}

\usepackage[
backend=biber,
style=alphabetic,
sorting=ynt
]{biblatex}
\title{Topology of the Bend Loci of Convex Piecewise Linear Functions}
\author{Jidong Wang\thanks{Department of Mathematics, University of Texas-Austin, Austin, Texas, 78712 USA. Email: jidongw@utexas.edu}}
\date{}

\begin{document}

\maketitle

\begin{abstract}
    This short article serves as the appendix for \cite{tran2022minimal}. We prove that a complete intersection of $n$ generic polyhedral hypersurfaces in $\R^d$ is $(d-n-1)$-connected for $d\geq 2, d>n$. 
\end{abstract}

\section{Background}

Any convex PL function $f:\R^d\to \R$ can be written in the following form
\begin{equation}\label{form}
    f(x)=\max\{\langle x,b_1\rangle +a_1,\cdots, \langle x,b_r\rangle + a_r\}
\end{equation}
for some $b_1,...,b_r\in \R^d$ and $a_1,...,a_r\in \R$.

\begin{definition} A \textit{polyhedral hypersurface} is the bend locus of any convex PL function. A polyhedral hypersurface is \textit{generic} if all the coefficients $b_i\in \R^d$ and $a_i\in \R$ come from probability distributions with continuous density. A \textit{complete intersection} of generic polyhedral hypersurfaces is the intersection of generic polyhedral hypersurfaces that come from independent probability distributions.
\end{definition}

In particular, with probability one, a complete intersection of generic polyhedral hypersurfaces is transverse and free of parallel faces. The main purpose of this article is to prove the following topological property about generic complete intersections.

\begin{thm}\label{thm:main}
Let $X$ be a complete intersection of $n$ generic polyhedral hypersurfaces in $\R^d$ where $d\geq 2, d>n$. Then $X$ and its one-point compactification are both $(d-n-1)$-connected. 
\end{thm}

Our proof is based on an argument from an unpublished note by Adiprasito \cite{adiprasito2020note}. For the reader's convenience, we provide complete details. 

\section{Proof outline}

\subsection{Induction base}

We first prove a special case which will serve as the base of an inductive argument.

\begin{prop}\label{prop:base}
Let $X$ be a polyhedral hypersurface in $\R^d$ ($d\geq 2$). Then 
\begin{enumerate}
    \item the one-point compactification of $X$ is $(d-2)$-connected.
    \item the union of compact faces os $X$ is $(d-2)$-connected.
\end{enumerate}

\end{prop}

\begin{proof}
     As CW-complexes, $\R^d\cup\{pt\}$ is obtained from $X\cup \{pt\}$ by attaching $d$-cells. Since $\R^d\cup\{pt\}\cong S^d$ is $(d-1)$-connected, $X\cup \{pt\}$ is $(d-2)$-connected. This proves the first statement. 
     
     Let $\overline{X}$ be the polyhedral subdivision of $\R^d$ induced by $X$. Let $\overline{X}^\text{cpt}$ be the union of all compact faces of $\overline{X}$ and $X^\text{cpt}$ be the union of all compact faces of $X$. We claim that $\R^d$ deformation retracts onto $\overline{X}^\text{cpt}$. First, pick any point $x\in \overline{X}^\text{cpt}$ and let $B$ be a sufficiently large open ball centered at $x$ that contains $\overline{X}^\text{cpt}$. Since $\R^d$ deformation retracts onto $B$, it remains to show that $B$ deformation retracts onto $\overline{X}^\text{cpt}$. This can be done using an argument in \cite{MR1867354} (Proof of Proposition 0.16. See also \Cref{fig:projection}). For each unbounded face $F$ of $\overline{X}$, $F\cap B$ is a polyhedron truncated by a large sphere. It can be thought of as a glass filled with water, the sphere being the surface of the water. Pick a point $O$ that is in $F$ and sufficiently far from the sphere. Then the radial projection from $O$ deformation retracts the water onto the glass. We repeat this process for $F'\cap B$ for all the unbounded lower dimensional faces $F'$ of $F$, until all the unbounded faces of $F$ are contracted. The overall effect is that we get a deformation retraction of $F\cap B$ on to the compact faces of $F$. Apply the above construction to $F\cap B$ for all unbounded faces $F$. This deformation retracts $B$ onto $\overline{X}^\text{cpt}$. Hence, $\R^d$ deformation retracts onto $\overline{X}^\text{cpt}$, so $\overline{X}^\text{cpt}$ is contractible. Since $\overline{X}^\text{cpt}$ is obtained from $X^\text{cpt}$ by attaching $d$-cells, $X^\text{cpt}$ is $(d-2)$-connected.
\end{proof}

\begin{figure}
    \centering
    \includegraphics[width=2in]{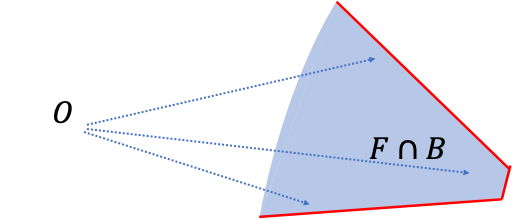}
    \caption{Radial projection from a distant point $O$.}
    \label{fig:projection}
\end{figure}

\subsection{Inductive steps}

We will prove \Cref{thm:main} and the following lemma in conjunction by inducting on both $d$ and $n$.

\begin{lemma}\label{lem:main}
    Let $X$ be a complete intersection of $n$ generic polyhedral hypersurfaces in $\R^d$ where $d\geq 2, d>n$. Let $P$ be a full-dimensional polyhedron defined by generic supporting hyperplanes. Then the pair $(X\cap P, X\cap \partial P)$ is $(d-n-1)$-connected.
\end{lemma}

To clean up notations, let $P^n_d$ be the statement of \Cref{thm:main} and $C^n_d$ be the statement of \Cref{lem:main} with parameters $d$ and $n$. \Cref{prop:base} says $P^1_d$ is true for all $d\geq 2$. If we can prove
\begin{equation}
    P^n_{d-1}\Rightarrow C^n_d \text{ and } P^{n-1}_d\wedge C^{n-1}_d \Rightarrow P^n_d,
\end{equation} 
then we are done, since with the above implications, 
\begin{equation}
    P^{n-1}_d\wedge P^{n-1}_{d-1} \Rightarrow P^n_d,
\end{equation}
which eventually reduces the question to $P^1_d$.

\begin{proof}[Proof of $ P^n_{d-1}\Rightarrow C^n_d$] Take $X$ and $P$ as in the hypothesis of $C^n_d$. Let $\cH$ be the minimal set of supporting hyperplanes of $P$. Consider the function
\begin{equation}
    f: P \mapsto \R, x\mapsto \prod_{H\in \cH} d(x,H)
\end{equation}
where $d(x,H)$ is the distance from $x$ to $H$.
The genericity condition for $X$ and $P$ guarantees that $X$ is a Whitney stratified space and that $f$ is a Morse function on $X$. Let $x_1,...,x_r$ be the finitely many critical points of $f_X$ and $t_1,...,t_r$ be the corresponding distinct critical values. Set
\begin{equation}
    X_{I} = f^{-1}(I)\cap X
\end{equation}
for $I$ an interval or a single point in $\R_{\geq 0}$. We keep track of the topology change of $X_{[0,t]}$ as $t$ crosses the critical values. There are two cases (see \Cref{fig:morse}).
    \begin{itemize}
        \item \textbf{Case 1}: $x_k$ is a vertex of $X$. Let $T_{x_k}X$ be the tangent fan of $X$ at $x_k$ and $H$ be the tangent hyperplane $T_{x_k}f^{-1}(t_k)$. Translate $H$ slightly towards $T_{x_k}f^{-1}([0,t_k))$ and call the new hyperplane $H'$. Let $A$ be the intersection of $H'$ with $T_{x_k}X$. Note that $A$ is an intersection of $n$ generic polyhedral hypersurfaces in $\R^{d-1}$.
        
        Let $\epsilon>0$ be sufficiently small. By definition, the space $X_{[0,t_k]}$ is obtained from $X_{[0,t_k-\epsilon)}$ by attaching $X_{(t_k-\epsilon,t_k]}$ along $f^{-1}(t_k-\epsilon)\cap X$, which may have (finitely) many connected components. Suppose $X_{(t_k-\epsilon,t_k]}$ and $f^{-1}(t_k-\epsilon)\cap X$ decompose into disjoint unions of their connected components as follows
        \begin{equation}
            X_{[t_k-\epsilon,t_k]} = \bigsqcup_i^\ell Y_i,\quad f^{-1}(t_k-\epsilon)\cap X = \bigsqcup_i^\ell Z_i,
        \end{equation}
        such that $Y_i$ is glued along $Z_i$ for each $i$. Suppose $x_k$ is in $Y_1$. Then $(Y_1,Z_1)$ is homotopy equivalent to $(CA,A)$, and $(Y_i,Z_i)$ is homotopy equivalent to $(\R^{d-n-1}\times [0,\epsilon],\R^{d-n-1})$. By the induction hypothesis, $A$ is $(d-n-2)$-connected. By the homotopy long exact sequence for the pair $(CA,A)$, $(CA,A)$ is $(d-n-1)$-connected. Hence, the gluing data $(X_{[t_k-\epsilon,t_k]},f^{-1}(t_k-\epsilon))$ is $(d-n-1)$-connected. By definition, $(X_{[0,t_k-\epsilon]},f^{-1}(t_k-\epsilon))$ is path-connected. By homotopy excision (\cite{MR1867354}, Theorem 4.23), $(X_{[0,t_k]},X_{[0,t_k-\epsilon]})$ is $(d-n-1)$-connected. 
        
        \item \textbf{Case 2}: $x_k$ is in the relative interior of a $k$-dimensional face $\sigma$ of $X$. In this case, the gluing data splits into normal Morse data and tangential Morse data. Similar to Case 1, the gluing data may have more than one connected components, and the only component that contributes to the topology change is the one where $x_k$ is. The treatment for multiple components is exactly the same as in the precious case. Therefore, it doesn't hurt to assume that the gluing data has only one connected component.
        
        The tangential Morse data is $(\sigma, \partial \sigma)$. Let $N$ be the orthogonal complement of $T_{x_k}\sigma$. The normal Morse data is 
         \begin{equation}
        (N\cap T_{x_i}f^{-1}([0,t_i])\cap T_{x_i}X, N\cap H'\cap T_{x_i}f^{-1}([0,t_i])\cap T_{x_i}X)
    \end{equation}
    where $H'$ is as in the previous case. Note that $N\cap T_{x_i}f^{-1}([0,t_i])\cap T_{x_i}X$ is the cone over $N\cap H'\cap T_{x_i}f^{-1}([0,t_i])\cap T_{x_i}X$. The latter is a complete intersection of $n$ generic polyhedral hypersurfaces in $\R^{d-k-1}$. By the induction hypothesis and the homotopy long exact sequence for a pair, the normal Morse data is $(d-k-n-1)$-connected. Since the tangential Morse data $(\sigma, \partial \sigma)$ is $(k-1)$-connected, the total gluing data
    \begin{equation}
        (\sigma,\partial \sigma )\times  (N\cap T_{x_i}f^{-1}([0,t_i])\cap T_{x_i}X, N\cap H'\cap T_{x_i}f^{-1}([0,t_i])\cap T_{x_i}X)
    \end{equation}
    is $(d-n-1)$-connected. Now repeat the argument in the previous case, we conclude that $(X_{[0,t_k]},X_{[0,t_k-\epsilon]})$ is $(d-n-1)$-connected.  
    
        \end{itemize}

        Observe that $X_{[0,t_{k+1}-\epsilon]}$ deformation retracts onto $X_{[0,t_k]}$, so to sum up, we have,
        \begin{equation}
            (X_{[0,t_{k+1}-\epsilon]},X_{[0,t_k-\epsilon]}) \text{ is $(d-n-1)$-connected}
        \end{equation}
        and
        \begin{equation}
            (X_{[0,t_1-\epsilon]},X_0) \text{ is $(d-n-1)$-connected}
        \end{equation}

        By applying homotopy long exact sequence to the triple $(X_{[0,t_{k+1}-\epsilon]},X_{[0,t_k-\epsilon]},X_0)$, the above two connectedness properties imply that $(X_{[0,t_{k+1}-\epsilon]},X_0)$ is $(d-n-1)$-connected. In other words, $(X_{[0,t)}, X_0)$ is $(d-n-1)$-connected for all $t>0$, which completes the proof.
\end{proof}
\begin{figure}[H]
    \centering
    \includegraphics[width=3.5in]{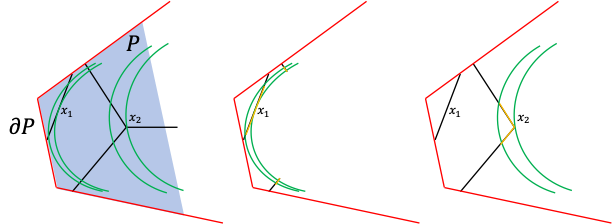}
    \caption{Illustration of the Morse-theoretical argument. $x_1$ is a critical point in the interior of a 1-cell. $x_2$ is a vertex critical point. The orange parts are the gluing data when the level set crosses each critical point.}
    \label{fig:morse}
\end{figure}

\begin{proof}[Proof of $P^{n-1}_d\wedge C^{n-1}_d \Rightarrow P^n_d$] Let $X$ be a complete intersection of $n-1$ generic polyhedral hypersurfaces and $X_n$ be another generic hypersurface. Let $X'=X\cap X_n$. Pick any connected component of $\R^d\backslash X_n$ and take its closure, which is a polyhedron $P$ with generic supporting hyperplanes. Applying \Cref{lem:main} to $X$ and $P$, we know that $X$ is obtained from $X\cap \partial P$ be gluing cells of dimension no less than $d-n+1$. Therefore, $X$ is obtained from $X'$ by gluing cells of dimension no less than $d-n+1$. Since by the induction hypothesis $X$ is $(d-n)$-connected, $X'$ is $(d-n-1)$-connected.
\end{proof}

\begin{rmk}
    The proof of the inductive steps is only given for $X$. The proof for the corresponding statement for the one-point compactification of $X$ only differs at the final stage of the Morse-theoretical argument in \Cref{lem:main}. Without the one-point compactification, the topology of $X_{[0,t)}$ no longer changes when $t$ exceeds the largest critical value, so $X_{[0,t)}$ is homotopy equivalent to $X\cap P$ for any $t>t_r$. With the one-point compactification, $X\cup \{pt\}$ is obtained from $X_{[0,t)}\cup \{pt\}$ by attaching more cells of dimension $d-n+1$, respectively, along their boundaries $X_t$. Therefore, \Cref{lem:main} also holds for $X\cup \{pt\}$.
\end{rmk}

\printbibliography

\end{document}